\documentclass[reqno,12pt]{amsart}
\usepackage{amsmath,amsthm,amssymb,amsfonts,amscd}
\usepackage{epsfig}
\numberwithin{equation}{section}

\theoremstyle{plain}
\newtheorem{theorem}{Theorem}

\newtheorem{corollary}[theorem]{Corollary}
\newtheorem{proposition}[theorem]{Proposition}
\newtheorem*{theorem*}{Theorem}
\newtheorem*{conjecture*}{Conjecture}

\theoremstyle{definition}
\newtheorem{remark}[theorem]{Remark}

\newtheorem*{definition}{Definition}

\newcommand{\CC}{{\mathbb{C}}}

\newcommand{\QQ}{{\mathbb{Q}}}

\newcommand{\ZZ}{{\mathbb{Z}}}

\def\C{{\mathcal C}}

\def\H{{\mathcal H}}

\begin{document}
\title[Orbifold E-functions]{Orbifold E-functions of dual invertible polynomials}
\author{Wolfgang Ebeling, Sabir M.~Gusein-Zade and Atsushi Takahashi}
\thanks{Partially supported 
by the DFG-programme SPP1388 ``Representation Theory'' and by a DFG-Mercator fellowship. The second named author is also supported by RFBR-13-01-00755
and NSh-5138.2014.1.
The third named author is also supported 
by Grant-in Aid for Scientific Research 
grant numbers 20360043 from the Ministry of Education, 
Culture, Sports, Science and Technology, Japan. 
}
\address{Institut f\"ur Algebraische Geometrie, Leibniz Universit\"at Hannover, Postfach 6009, D-30060 Hannover, Germany}
\email{ebeling@math.uni-hannover.de}
\address{Moscow State University, Faculty of Mechanics and Mathematics,\\
Moscow, GSP-1, 119991, Russia}
\email{sabir@mccme.ru}
\address{Department of Mathematics, Graduate School of Science, Osaka University, 
Toyonaka Osaka, 560-0043, Japan}
\email{takahashi@math.sci.osaka-u.ac.jp}
\begin{abstract} 
An invertible polynomial is a quasihomogeneous polynomial with the number of monomials coinciding with the
number of variables and such that the weights of the variables and the quasi-degree are well defined. 
In the framework of the search for mirror symmetric orbifold Landau--Ginzburg models,
P.~Berg\-lund and M.~Henningson considered a pair $(f,G)$ consisting of
an invertible polynomial $f$ and an abelian group $G$ of its symmetries together with
a dual pair $(\widetilde{f}, \widetilde{G})$. We consider the so-called orbifold E-function of such a pair  $(f,G)$ which is a generating function for the exponents of the monodromy action on an orbifold version of the mixed Hodge structure on the Milnor fibre of $f$. We prove that the orbifold E-functions of Berglund--Henningson dual pairs coincide up to a sign depending on the number of variables. The proof is based on a relation between
monomials (say, elements of a monomial basis of the Milnor algebra of an invertible polynomial)
and elements of the whole symmetry group of the dual polynomial.
\end{abstract}
\maketitle

\section*{Introduction} 
Mirror symmetry is the famous observation by physicists that there exist pairs of Calabi--Yau manifolds whose Hodge diamonds are symmetric in a certain sense. P.~Berglund and T.~H\"ubsch \cite{BHu} suggested a method to construct such 
mirror symmetric pairs of Calabi--Yau manifolds. 
They considered a polynomial $f$ of a special form, a so called  {\em invertible} one, and its {\em Berglund--H\"ubsch transpose} $\widetilde f$: see below. The manifolds suggested in \cite{BHu} were resolutions of the following ones: the hypersurface in a weighted projective space defined by the equation $f=0$ and the quotient of the hypersurface in a weighted projective space defined by the equation $\widetilde{f}=0$ by a certain group of symmetries of the polynomial $\widetilde f$.
 In \cite{BHu} these polynomials
appeared as potentials of Landau--Ginzburg models.
In \cite{BHe} this construction was generalized to an orbifold setting. An orbifold
Landau--Ginzburg model of \cite{BHe} was described by a pair $(f, G)$, where $f$ is
an  invertible polynomial and $G$ is a (finite) abelian group of symmetries of $f$.
For a pair $(f, G)$ they defined the dual pair $(\widetilde{f}, \widetilde{G})$. There were
observed some symmetries of the manifolds constructed in this way. Namely, it was observed that the elliptic genera of dual pairs coincide up to sign \cite{BHe, KY}. In \cite{BHe}, there was
defined the Poincar\'e polynomial of a pair $(f, G)$ as a particular limit of the
elliptic genus of the corresponding orbifold. It is a fractional power polynomial in
one variable. The paper \cite{BHe} contains a physical style proof of the fact that
the Poincar\'e polynomials of dual pairs coincide up to sign. 

Certain aspects of Landau--Ginzburg models are related to singularity theory.
In \cite{ET1} (see also \cite{ET2}), for a pair $(f, G)$ with the group $G$ satisfying some restrictions,
there was defined a so called orbifold E-function of the pair which is a generating function for the exponents of the monodromy action on an orbifold version of the mixed Hodge structure on the Milnor fibre of $f$. It turns out that it agrees formally with a two-variables generalization
of the Poincar\'e polynomial. There exists a conjectural symmetry property of it
for Berglund--Henningson dual pairs (see Theorem~\ref{main}). It seems that this symmetry property was understandable for specialists in Landau--Ginzburg models, however to our knowledge there did not exist a proof of it.
Here we give a simple and readable proof of this symmetry property and therefore also a new proof of the statement of P.~Berg\-lund and M.~Henningson. The method of the proof is different
from that for the Poincar\'e polynomial in \cite{BHe}. It is based on a relation between
monomials (say, elements of a monomial basis of the Milnor algebra of an invertible polynomial)
and elements of the maximal symmetry group of the dual polynomial. This relation was
essentially described in \cite{Kreuzer}. It was generalized by M.~Krawitz \cite{Krawitz} to monomials invariant under a so called admissible subgroup $G$ of the maximal symmetry  group $G_f$ of $f$ and its dual group $\widetilde{G}$ (see also \cite[Proposition~3.6]{ET1}). Our proof of the symmetry property applies to any subgroup $G$ of the maximal symmetry group $G_f$ of $f$, not only to admissible ones. Moreover, as an intermediate step we derive a formula for the orbifold E-function expressing it in a symmetric way in terms of the group and the dual group (see Proposition~\ref{prop:formula}). 

As a corollary we get the statement obtained earlier
that the reduced orbifold zeta functions of dual pairs either coincide or are inverse to each other 
(depending on the number of variables): \cite{EG2}.

\section{Invertible polynomials}
Let $f(x_1,\dots, x_n)$ be a non-degenerate weighted homogeneous polynomial, namely, a polynomial with an isolated singularity at the origin with the property that there are positive integers $w_1,\dots ,w_n$ and $d$ such that 
$f(\lambda^{w_1} x_1, \dots, \lambda^{w_n} x_n) = \lambda^d f(x_1,\dots ,x_n)$ 
for $\lambda \in \CC^\ast$. We call $(w_1,\dots ,w_n;d)$ a system of {\em weights}. 
\begin{definition}
A non-degenerate weighted homogeneous polynomial $f(x_1,\dots ,x_n)$ is called {\em invertible} if 
the following conditions are satisfied:
\begin{enumerate}
\item the number of variables ($=n$) coincides with the number of monomials 
in the polynomial $f(x_1,\dots x_n)$, 
namely, 
\[
f(x_1,\dots ,x_n)=\sum_{i=1}^na_i\prod_{j=1}^nx_j^{E_{ij}}
\]
for some coefficients $a_i\in\CC^\ast$ and non-negative integers 
$E_{ij}$ for $i,j=1,\dots, n$,
\item a system of weights $(w_1,\dots ,w_n;d)$ can be uniquely determined by 
the polynomial $f(x_1,\dots ,x_n)$ up to a constant factor ${\rm gcd}(w_1,\dots ,w_n;d)$, 
namely, the matrix $E:=(E_{ij})$ is invertible over $\QQ$.
\end{enumerate}
\end{definition}

According to \cite{KS}, an invertible polynomial $f$ is a (Thom-Sebastiani) sum of invertible polynomials (in groups of different variables) of the following types:
\begin{enumerate}
\item[1)] $x_1^{a_1}x_2 + x_2^{a_2}x_3 + \ldots + x_{m-1}^{a_{m-1}}x_m + x_m^{a_m}$ (chain type; $m\ge 1$);
\item[2)] $x_1^{a_1}x_2 + x_2^{a_2}x_3 + \ldots + x_{m-1}^{a_{m-1}}x_m + x_m^{a_m}x_1$ (loop type; $m\ge 2$).
\end{enumerate}
\begin{remark}
In \cite{KS} the authors distinguished also polynomials of the so called Fermat type: $x_1^{a_1}$. One can regard the Fermat type polynomial $x_1^{a_1}$ as a chain type polynomial with $m=1$.
\end{remark}

\begin{definition}
Let $f(x_1,\dots ,x_n)=\sum_{i=1}^na_i\prod_{j=1}^nx_j^{E_{ij}}$ be an invertible polynomial. Without loss of generality, we may and will assume in the sequel that $\det E > 0$.
Define rational numbers $q_1,\dots, q_n$ by the unique solution of the equation
\begin{equation*}
E
\begin{pmatrix}
q_1\\
\vdots\\
q_n
\end{pmatrix}
=
\begin{pmatrix}
1\\
\vdots\\
1
\end{pmatrix}
.
\end{equation*}
Namely, set $q_i := w_i/d$, $i=1, \ldots , n$ for the system of weights $(w_1,\dots ,w_n;d)$.
\end{definition}

\begin{definition}
Let $f(x_1,\dots ,x_n)=\sum_{i=1}^na_i\prod_{j=1}^nx_j^{E_{ij}}$ be an invertible polynomial.
The {\em group of diagonal symmetries} $G_f$ of $f$ is
the finite abelian group defined by 
\[
G_f=\left\{(\lambda_1,\dots ,\lambda_n)\in (\CC^\ast)^n\, \left| \, \prod_{j=1}^n \lambda_j ^{E_{1j}}=\dots =\prod_{j=1}^n\lambda_j^{E_{nj}}=1 \right\} \right..
\]
\end{definition}
Note that the polynomial $f$ is invariant with respect to the natural action of 
$G_f$ on the variables. Namely, we have 
\[
f(\lambda_1 x_1, \dots, \lambda_n x_n) = f(x_1,\dots ,x_n)
\]
for $(\lambda_1,\dots ,\lambda_n)\in G_f$. Note that $G_f$ always contains the exponential grading operator 
\[
g_0:=({\bf e}[q_1], \ldots , {\bf e}[q_n]),
\]
where ${\bf e}[ - ] = e^{2 \pi \sqrt{-1} \cdot  -}$. Denote by $G_0$ the subgroup of $G_f$ generated by $g_0$.

Let $f(x_1, \ldots , x_n)$ be an invertible polynomial and $G$ a subgroup of $G_f$
For $g \in G$, we denote by ${\rm Fix}\, g :=\{x\in\CC^n~|~g\cdot x=x \}$ the fixed locus of $g$, 
by $n_g: = \dim {\rm Fix}\, g$ its dimension and by $f^g:=f|_{{\rm Fix}\, g}$ the restriction of $f$ to the fixed locus of $g$. 
Note that the function $f^g$ has an isolated singularity at the origin \cite[Proposition~5]{ET2}.

We introduce the notion of the age of an element of a finite group as follows: 
\begin{definition}[Ito--Reid \cite{IR}]
Let $g \in G$ be an element and $r$ be the order of $g$. Then $g$ has a unique expression of the following form  
\[
g=\left({\bf e}\left[\frac{a_1}{r}\right], \ldots, {\bf e}\left[\frac{a_n}{r}\right]\right) \quad \mbox{with } 0 \leq a_i < r.
\]
Such an element $g$ is often simply denoted by $g=\frac{1}{r}(a_1, \ldots, a_n)$. The {\em age} of $g$ is defined as 
\[ {\rm age}(g) := \frac{1}{r}\sum_{i=1}^n a_i. \] 
\end{definition}

Berglund and H\"{u}bsch \cite{BHu} defined the dual polynomial $\widetilde{f}$ of $f$ as follows.
\begin{definition}
Let $f(x_1,\dots ,x_n)$ be an invertible polynomial. 
The {\em Berglund--H\"{u}bsch transpose} $\widetilde{f}(x_1,\dots ,x_n)$ 
of $f(x_1,\dots ,x_n)$ is defined by
\begin{equation}
\widetilde{f}(x_1,\dots ,x_n):=\sum_{i=1}^na_i\prod_{j=1}^nx_j^{E_{ji}}.
\end{equation}
\end{definition}

Berglund and Henningson \cite{BHe} defined the dual group $\widetilde{G}$ of $G$ as follows.
\begin{definition} The {\em dual group} $\widetilde{G}$ is defined by
\begin{equation}
\widetilde{G} := {\rm Hom}(G_f/G, \CC^\ast).
\end{equation}
\end{definition}

We recall here some important properties of dual groups. 
\begin{proposition}[cf. Ebeling--Takahashi~\cite{ET2}, Krawitz~\cite{Krawitz}]
Let $f(x_1, \ldots, x_n)$ be an invertible polynomial.
\begin{enumerate}
\item There is a natural group isomorphism $G_{\widetilde{f}}\cong {\rm Hom}(G_f, \CC^\ast)$.
\item Let $G \subset G_f$ be a subgroup. One has ${\rm Hom}(G_{\widetilde{f}}/\widetilde{G}, \CC^\ast)=G$.
\item There is a natural group isomorphism $\widetilde{G_0}\cong {\rm SL}(n;\CC)\cap G_{\widetilde{f}}$.
\end{enumerate}
\end{proposition}

\section{Orbifold E-functions}
Let $f(x_1,\dots, x_n)$ be an invertible polynomial.
Steenbrink \cite{st:1} constructed a canonical mixed Hodge structure on the vanishing cohomology $H^{n-1}(Y_\infty,\CC)$ of $f$
with an automorphism $c$ given by the Milnor monodromy. 
In our situation, the automorphism $c$ coincides with the one induced by the action of $g_0\in G_f$ on $\CC^n$, 
which is semi-simple.
For $\lambda\in \CC$, let  
\begin{equation}
H^{n-1}(Y_\infty,\CC)_\lambda:={\rm Ker}(c-\lambda\cdot {\rm id}:H^{n-1}(Y_\infty,\CC)\longrightarrow 
H^{n-1}(Y_\infty,\CC)).
\end{equation}
Denote by $F^\bullet$ the Hodge filtration of the mixed Hodge structure.

We can naturally associate a $\QQ\times \QQ$-graded vector space to the mixed Hodge structure 
with an automorphism. 
\begin{definition}
Define the $\QQ\times \QQ$-graded vector space $\H_f:=\displaystyle\bigoplus_{p,q\in\QQ}\H^{p,q}_f$ as
\begin{enumerate}
\item If $p+q\ne n$, then  $\H^{p,q}_f:=0$.
\item If $p+q=n$ and $p\in\ZZ$, then  
\[
\H^{p,q}_f:={\rm Gr}^{p}_{F^\bullet}H^{n-1}(Y_\infty,\CC)_1.
\]
\item If $p+q=n$ and $p\notin\ZZ$, then  
\[
\H^{p,q}_f:={\rm Gr}^{[p]}_{F^\bullet}H^{n-1}(Y_\infty,\CC)_{e^{2\pi\sqrt{-1} p}},
\]
where $[p]$ is the largest integer less than $p$.
\end{enumerate}
\end{definition}

Let $f(x_1, \ldots , x_n)$ be an invertible polynomial and $G$ a subgroup of $G_f$.
We shall use the fact that $\H_{f^g}$ admits a natural $G$-action 
by restricting the $G$-action on $\CC^n$ to ${\rm Fix}\, g$ (which is well-defined since $G$ acts diagonally on $\CC^n$).

To the pair $(f,G)$ we can associate the following $\QQ\times \QQ$-graded super vector space:
\begin{definition}
Define the $\QQ\times \QQ$-graded super vector space $\H_{f,G}:=\H^{p,q}_{f,G,\bar{0}}\oplus \H^{p,q}_{f,G,\bar{1}}$ as 
\begin{equation}
\H_{f,G,\overline{0}}:=\bigoplus_{\substack{g\in G;\\ n_g\equiv 0\ (\text{\rm mod } 2)}}(\H_{f^g})^G(-{\rm age}(g),-{\rm age}(g)),
\end{equation}
\begin{equation}
\H_{f,G,\overline{1}}:=\bigoplus_{\substack{g\in G;\\ n_g\equiv 1\ (\text{\rm mod } 2)}}(\H_{f^g})^G(-{\rm age}(g),-{\rm age}(g)),
\end{equation}
where $(\H_{f^g})^G$ denotes the $G$-invariant subspace of $\H_{f^g}$.
\end{definition}
\begin{definition}
The {\em E-function} for the pair $(f,G)$ is 
\begin{equation}
E(f,G)(t,\bar{t})=\displaystyle\sum_{p,q\in\QQ}\left({\rm dim}_\CC (\H_{f,G,\bar{0}})^{p,q}-{\rm dim}_\CC (\H_{f,G,\bar{1}})^{p,q}\right)
\cdot t^{p-\frac{n}{2}}{\bar{t}}^{q-\frac{n}{2}}.
\end{equation}
\end{definition}
In general, we may have both $(\H_{f,G,\overline{0}})^{p,q}\ne 0$ and $(\H_{f,G,\overline{1}})^{p,q}\ne 0$ for some $p,q\in\QQ$.
Indeed, the pair $(f,G)=(x_1^4+x_2^4, \langle \frac{1}{4}(1,1)\rangle)$ gives such an example, namely, $(\H_{f,G,\overline{0}})^{1,1}$ and 
$(\H_{f,G,\overline{1}})^{1,1}$ are both one dimensional.

\begin{proposition}
Let $f(x_1, \ldots , x_n)$ be an invertible polynomial and $G$ a subgroup of $G_f$.
Assume $G\subset {\rm SL}(n;\CC)$ or $G\supset G_0$.
If $(\H_{f,G,\overline{i}})^{p,q}\ne 0$, then $(\H_{f,G,\overline{i+1}})^{p,q}= 0$ for all $p,q\in\QQ$ and $i\in\ZZ/2\ZZ$.
\end{proposition}
\begin{proof}
Note that $(\H_{f,G,\overline{i}})^{p,q}\ne 0$ if only if 
there exists $g\in G$ with $n_g\equiv i$ $(\text{\rm mod } 2)$ such that $(\H^{p-{\rm age}(g),q-{\rm age}(g)}_{f^g})^G\ne 0$.
In particular, we have $n_g=({\rm age}(g)-p)+({\rm age}(g)-q)$ and $n_g\equiv i$ $(\text{\rm mod } 2)$ for such $g$.
If the group $G$ is a subgroup of ${\rm SL}(n;\CC)$, then ${\rm age}(g)\in\ZZ$ for all $g\in G$. 
Therefore, there does not exist $g'\in G$ such that $n_{g'}=({\rm age}(g')-p)+({\rm age}(g')-q)$ and $n_{g'}\equiv i+1$ $(\text{\rm mod } 2)$ since $n_{g'}=2\cdot {\rm age}(g')-p-q\equiv -p-q\equiv n_{g}\equiv i$ $(\text{\rm mod } 2)$.
If the group $G$ contains $G_0$, then $p-{\rm age}(g), q-{\rm age}(g)\in\ZZ$ for all $g\in G$ 
since we must have $(\H^{p-{\rm age}(g),q-{\rm age}(g)}_{f^g})^{G_0}\ne 0$ which means the Milnor monodromy, 
identified with the automorphism induced by $g_0$, acts trivially on this vector space. 
Therefore, there does not exist $g'\in G$ such that $n_{g'}=({\rm age}(g')-p)+({\rm age}(g')-q)$ and $n_{g'}\equiv i+1$ $(\text{\rm mod } 2)$ since $n_{g'}=({\rm age}(g')-p)+({\rm age}(g')-q)=({\rm age}(g)-p)+({\rm age}(g)-q)+2({\rm age}(g')-{\rm age}(g))
=n_g+2({\rm age}(g')-{\rm age}(g))\equiv n_{g}\equiv i$ $(\text{\rm mod } 2)$.
\end{proof}

\begin{definition}
Let $f(x_1, \ldots , x_n)$ be an invertible polynomial and $G$ a subgroup of $G_f$.
Assume $G\subset {\rm SL}(n;\CC)$ or $G\supset G_0$.
The {\em Hodge numbers} for the pair $(f,G)$ are
\begin{equation}
h^{p,q}(f,G):=\dim_\CC (\H_{f,G,\overline{0}})^{p,q}+\dim_\CC (\H_{f,G,\overline{1}})^{p,q},\quad p,q\in\QQ.
\end{equation}
\end{definition}

\begin{proposition}\label{prop4}
Let $f(x_1, \ldots , x_n)$ be an invertible polynomial and $G$ a subgroup of $G_f$.
The E-function is given by 
\begin{equation}
E(f,G)(t,\bar{t}):=
\begin{cases}
\displaystyle\sum_{p,q\in\QQ}(-1)^{p+q}h^{p,q}(f,G) \cdot 
t^{p-\frac{n}{2}}\bar{t}^{q-\frac{n}{2}},\ \text{if }G \subset {\rm SL}_n(\CC),\\
\displaystyle\sum_{p,q\in\QQ}(-1)^{-p+q}h^{p,q}(f,G) \cdot 
t^{p-\frac{n}{2}}\bar{t}^{q-\frac{n}{2}}, \ \text{if } G_0 \subset G.
\end{cases}
\end{equation}
\end{proposition}
\begin{proof}
If $G\subset{\rm SL}(n;\CC)$, then we have $n_g=2\cdot {\rm age}(g)-p-q\equiv p+q$  $(\text{\rm mod } 2)$ for $g\in G$ with $n_g\equiv i$ $(\text{\rm mod } 2)$ such that $(\H^{p-{\rm age}(g),q-{\rm age}(g)}_{f^g})^G\ne 0$, which implies the first statement.
If $G\supset G_0$, then we have $n_g=({\rm age}(g)-p)+({\rm age}(g)-q)=-p+q+2({\rm age}(g)-q)\equiv -p+q$  $(\text{\rm mod } 2)$ for $g\in G$ with $n_g\equiv i$ $(\text{\rm mod } 2)$ such that $(\H^{p-{\rm age}(g),q-{\rm age}(g)}_{f^g})^G\ne 0$, which implies the second statement.
\end{proof}
\begin{remark}
Note that if $G_0\subset G\subset {\rm SL}_n(\CC)$, then both descriptions of E-functions coincide, 
more precisely, one has $h^{p,q}(f,G)=0$ for $p,q\notin \ZZ$.
\end{remark}
\begin{proposition}[Steenbrink \cite{st:1}]
Let $f(x_1, \ldots , x_n)$ be an invertible polynomial.
One has 
\[
E(f,\{1\})(t,\bar{t})=(-1)^n\prod_{i=1}^n\frac{1-\left(\frac{\bar{t}}{t}\right)^{1-q_i}}{1-\left(\frac{\bar{t}}{t}\right)^{q_i}}\left(\frac{\bar{t}}{t}\right)^{q_i-\frac{1}{2}}.
\]
\end{proposition}
\begin{remark}\label{rem:Poincare}
Note that $E(f,\{1\})(1,y)$ is the Poincar\'e polynomial of $\Omega_f$ in $y$ 
with respect to the degrees $q_1,\dots, q_n$ where $\Omega_f:= \Omega^n(\CC^n)/(df \wedge \Omega^{n-1}(\CC^n))$ and $\Omega^p(\CC^n)$ is the space of regular $p$-forms on $\CC^n$. 
\end{remark}
\begin{proposition}\label{formula}
Let $f(x_1, \ldots , x_n)$ be an invertible polynomial and $G$ a subgroup of $G_f$.
Write $g \in G$ in the form 
$(\lambda_1(g), \ldots, \lambda_n(g))$ where $\lambda_i(g)={\bf e}[a_iq_i]$. The E-function for the pair $(f,G)$ is naturally given by the following formula$:$ 
\begin{equation}
E(f,G)(t,\bar{t})=\sum_{g\in G}E_g(f,G)
(t,\bar{t}),
\end{equation}
\begin{equation}
E_g(f,G)(t,\bar{t})
:=\left(\prod_{q_ia_i \not\in \ZZ}\left({t}{\bar{t}}\right)^{q_ia_i-[q_ia_i]-\frac{1}{2}} \right)
\cdot (-1)^{n_g} \frac{1}{|G|} \sum_{h\in G}\prod_{q_ia_i\in\ZZ}
\frac{\lambda_i(h)\left(\frac{\bar{t}}{t}\right)^{q_i}-\left(\frac{\bar{t}}{t}\right)}{1-\lambda_i(h)\left(\frac{\bar{t}}{t}\right)^{q_i}}\left(\frac{\bar{t}}{t}\right)^{-\frac{1}{2}}.
\end{equation}
Here $[a]$ for $a \in \QQ$ denotes the largest integer less than $a$.
\end{proposition}
\begin{proof}
This is obtained from the definition of $E_g(f,G)(t,\bar{t})$ by a straightforward calculation. 
See \cite[Theorem~6]{ET2}, for example.
\end{proof}

\section{Main theorem}
\begin{theorem}\label{main}
Let $f(x_1, \ldots , x_n)$ be an invertible polynomial and $G$ a subgroup of $G_f$. Then 
\begin{equation}
E(f,G)(t, \bar{t}) = (-1)^n  E(\widetilde{f},\widetilde{G})(t^{-1},\bar{t}). 
\end{equation}
\end{theorem}

\begin{corollary}
Let $f(x_1, \ldots , x_n)$ be an invertible polynomial and $G$ a subgroup of $G_f\cap {\rm SL}(n;\CC)$. Then 
for all $p,q\in\QQ$, one has
\begin{equation}
h^{p,q}(f,G)=h^{n-p,q}(\widetilde{f},\widetilde{G}).
\end{equation}
\end{corollary}
\begin{proof}
This follows from Proposition~\ref{prop4}.
\end{proof}

Let $f(x_1, \ldots , x_n)$ be an invertible polynomial and $G$ a subgroup of $G_f$ containing $G_0$. 
One can define an {\em exponent} of the pair $(f,G)$ as the rational number $q$ with $h^{p,q}(f,G)\ne 0$. 
The {\em set of exponents} of the pair $(f,G)$ is the multi-set of exponents 
\[
\left\{q*h^{p,q}(f,G)~|~p,q\in\QQ,\ h^{p,q}(f,G)\ne 0 \right\},
\]
where by $u*v$ we denote $v$ copies of the rational number $u$.  
It is almost clear that the mean of the set of exponents of $(f,G)$ is $n/2$, namely, we have 
\[
\sum_{p,q\in\QQ}(-1)^{-p+q} \left(q-\frac{n}{2} \right)h^{p,q}(f,G)=0.
\]
It is natural to ask what is the {\em variance of the set of exponents} of $(f,G)$ defined by 
\[ {\rm Var}_{(f,G)} := \sum_{p,q\in\QQ}(-1)^{-p+q} \left(q- \frac{n}{2} \right)^2 h^{p,q}(f,G). \]
\begin{corollary}
One has
\begin{equation}
{\rm Var}_{(f,G)} = \frac{1}{12} \hat{c} \cdot\chi(f,G),
\end{equation}
where $\hat{c} := n - 2\sum_{i=1}^n q_i$ and $\chi(f,G):=E(f,G)(1,1)$.
\end{corollary}
\begin{proof}
This follows from Theorem~\ref{main} and \cite[Theorem~19]{ET2} since $\widetilde{G}\subset {\rm SL}(n;\CC)$. 
\end{proof}

\section{Proof}
For the proof of Theorem~\ref{main}, we need some preparations.
Let $f=f_1\oplus\ldots\oplus f_s$ be an invertible polynomial such that each $f_i$, $i=1,\ldots, s$,
is either of chain or of loop type where $\oplus$ means the Thom--Sebastiani sum.
Let $A_f=\CC[x_1, \ldots, x_n]\left/\langle\partial f/\partial x_1,\ldots, \partial f/\partial x_n\rangle\right.$ be 
the Milnor algebra of $f$, $A_f=A_{f_1}\otimes\dots\otimes A_{f_s}$.

Let $\Omega^p(\CC^n)$ be the $\CC$-vector space of regular $p$-forms on $\CC^n$. 
Consider the free $A_f$-module $\Omega_f:= \Omega^n(\CC^n)/(df \wedge \Omega^{n-1}(\CC^n))$ of rank one. 
Note that $\Omega_f\cong \Omega_{f_1}\otimes \dots \otimes \Omega_{f_s}$.
There is a mapping $\psi:\Omega_f\longrightarrow \CC[G_{\widetilde{f}}]$ defined as follows. 

For an $n$-form ${\overline{x}}^{\overline{k}}dx_1\wedge\dots\wedge dx_n
=x_1^{k_1}\cdot\dots\cdot x_n^{k_n}dx_1\wedge\dots\wedge dx_n$, let
\begin{equation}\label{beta}
 (\beta_1, \dots, \beta_N):=(k_1+1, \dots, k_n+1)E^{-1},
\end{equation}
where $E$ is the matrix of the exponents of $f$. 
It follows that the element $({\mathbf e}[\beta_1], \ldots, {\mathbf e}[\beta_n])\in(\CC^*)^n$ belongs to $G_{\widetilde{f}}$. 
Set $\psi({\overline{x}}^{\overline{k}}dx_1\wedge\dots\wedge dx_n):=({\mathbf e}[\beta_1], \dots, {\mathbf e}[\beta_n])$.
 
A monomial basis of $\Omega_f$ can be chosen as the product of monomial bases of $\Omega_{f_i}$ for $i=1, \dots, s$
and the map $\psi$ acts on the factors independently.

\begin{proposition}\label{Omega to G}
 There exists a monomial basis of $\Omega_{f_i}$ such that
 \begin{enumerate}
  \item[1)] if $f_i$ is of chain type, the map $\psi$ is a one-to-one map of this basis onto the set
  of elements $\widetilde{g}_i\in G_{\widetilde{f}_i}$ preserving an even number of the variables; 
  \item[2)] if $f_i$ is of loop type with an odd number of variables, the map $\psi$ is a one-to-one map of this basis onto
  $G_{\widetilde{f}_i}\setminus\{1\}$;
  \item[3)] if $f_i$ is of loop type with an even number of variables, the map $\psi$ maps this basis onto
  $G_{\widetilde{f}_i}$ so that each element except the unit $1\in G_{\widetilde{f}_i}$ has one preimage,
  but $1$ has two preimages.
 \end{enumerate}
 The degree $\ell(\overline{k}):=\sum_{i=1}^n q_i(k_i+1)$ of the form ${\overline{x}}^{\overline{k}}dx_1\wedge\ldots\wedge dx_n$ 
of the basis is equal to ${\rm age}(\psi(\overline{k}))+\frac{n_{\psi(\overline{k})}}{2}$, where $n_{\widetilde{g}}$,
 $\widetilde{g}\in G_{\widetilde{f}}$, is the number of the coordinates fixed by the element $\widetilde{g}$.
 \end{proposition}

 \begin{remark}
 For loop type polynomials this means that the quasidegree is equal either to ${\rm age}(\psi(\overline{k}))$ (if $\psi(\overline{k}) \neq 1$) or to the number of variables divided by 2 (if $\psi(\overline{k}) = 1$). The latter only happens for an even number of variables.
 \end{remark}
 
 \begin{proof}
 We shall use the monomial basis of the Milnor algebra $A_{f_i}$ introduced in \cite{Kreuzer} in the form
 described in \cite[Lemma~1.5]{Krawitz}.
 
 Let $f_i$ be of chain type: $f_i(x_1,\ldots,x_N)=x_1^{a_1}x_2+x_2^{a_2}x_3+\ldots+x_{N-1}^{a_{N-1}}x_N+x_N^{a_N}$.
 The monomial basis described in \cite{Kreuzer} consists of all the monomials ${\overline{x}}^{\overline{k}}$,
 $\overline{k}=(k_1, \ldots,k_N)$, such that
 \begin{enumerate}
  \item[1)] $0\le k_m\le a_m-1$,
  \item[2)] if
  $$
  k_m=
  \begin{cases}
   a_m-1 {\text{ for all odd $m$, $m\le 2s-1$,}},\\
   0 {\text{ for all even $m$, $m\le 2s-1$,}},
  \end{cases}
  $$
  then $k_{2s}=0$.
 \end{enumerate}
 Equation (\ref{beta}) gives
 \begin{eqnarray}
  \beta_1&=&\frac{(k_1+1)}{a_1},\nonumber\\
  \beta_2&=&\frac{(k_2+1)a_1-(k_1+1)}{a_1a_2},\label{beta2}\\
  \beta_m&=&\frac{(k_m+1)-\beta_{m-1}}{a_m}\quad\text{for $2\le m\le N$}.\nonumber
 \end{eqnarray}
 One can see that $0<\beta_1\le 1$. If $\beta_1<1$, them $0<\beta_m<1$ for all $m=2,\ldots,N$.
 In this case no coordinates are fixed by $\psi(\overline{k})$ and
 ${\rm age}(\psi(\overline{k}))=\sum_m\beta_m=\sum_mq_m(k_m+1)=\ell(\overline{k})$.
 
 If $\beta_1=1$, i.e. $k_1+1=a_1$, then $k_2=0$ (due to the description of the monomial basis)
 and the equations for $\beta_3$, \dots, $\beta_N$ coincide with the equations like (\ref{beta2})
 written for $f_i'(x_3,\ldots,x_N)=x_3^{a_3}x_4+x_4^{a_4}x_5+\ldots+x_{N-1}^{a_{N-1}}x_N+x_N^{a_N}$.
 This permits to repeat the above arguments. They imply that we arrive at the following situation:
 there exists $s\ge 1$ such that $\beta_1=\beta_3=\ldots=\beta_{2s-1}=1$,
 $\beta_2=\beta_4=\ldots=\beta_{2s}=0$, $0<\beta_m<1$ for $m>2s$. The element $\psi(\overline{k})$
 fixes $2s$ coordinates: $x_1$, \dots, $x_{2s}$, ${\rm age(\psi(\overline{k}))}=\sum\limits_{m=2s+1}^{N}\beta_m$,
 $\ell(\overline{k})=\sum\limits_{m=1}^{N}\beta_m=s+{\rm age(\psi(\overline{k}))}$.
 This completes the proof for the chain case.
 
 Let $f_i$ be of loop type: $f_i(x_1,\ldots,x_N)=x_1^{a_1}x_2+x_2^{a_2}x_3+\ldots+x_{N-1}^{a_{N-1}}x_N+x_N^{a_N}x_1$.
 The monomial basis of $A_{f_i}$ consists of all the monomials ${\overline{x}}^{\overline{k}}$
 with $0\le k_m\le a_m-1$. If an element of $G_{\widetilde{f}_i}$ fixes a coordinate, then it fixes all
 the coordinates, i.e. it is equal to $1$. Equation~(\ref{beta2}) gives:
 $$
 \beta_N=\frac{\sum\limits_{j=1}^N a_1\ldots a_{N-j}(k_{N-j+1}+1)}{a_1\ldots a_N+(-1)^{N+1}}\,.
 $$
 If the number of variables $N$ is odd, $N=2M+1$, one has
 $$
 \beta_N=
 \frac{a_1\ldots a_{N-1}(k_N+1)-\sum_{j=1}^M\{a_1\ldots a_{N-2j}(k_{N-2j+1}+1)-a_1\ldots a_{N-2j-1}(k_{N-2j}+1)\}}
 {a_1\ldots a_N+1}\,.
 $$
 Each summand in the curly brackets is non-negative and $a_1\ldots a_{N-1}(k_{N}+1)\le a_1\ldots a_N$.
 Therefore $\beta_N<1$. Also one has
 $$
 \beta_N=
 \frac{\sum_{j=1}^M\{a_1\ldots a_{N-2j+1}(k_{N-2j+2}+1)-a_1\ldots a_{N-2j}(k_{N-2j+1}+1)\}+(k_1+1)}
 {a_1\ldots a_N+1}\,.
 $$
 Each summand in the curly brackets is non-negative and $k_1+1>0$.
 Therefore $\beta_N>0$. These arguments apply to each of the coordinates (using a cyclic permutation)
 and therefore one has $0<\beta_m<1$ for all $m$. Thus in this case the element $\psi(\overline{k})$
 fixes no coordinates and ${\rm age}(\psi(\overline{k}))=\sum\limits_{m=1}^{N}\beta_m=\ell(\overline{k})$.
 
 If $N$ is even,  $N=2M$, one has
 $$
 \beta_N=
 \frac{\sum_{j=1}^M\{a_1\ldots a_{N-2j+1}(k_{N-2j+2}+1)-a_1\ldots a_{N-2j}(k_{N-2j+1}+1)\}}
 {a_1\ldots a_N-1}\,.
 $$
 Each summand in the curly brackets is non-negative and therefore $\beta_N\ge 0$.
 It is equal to zero if and only if $k_{2m}=0$, $k_{2m-1}+1=a_{2m-1}$ for all $m=1,\dots, M$.
 In this case $\beta_{2m}=0$, $\beta_{2m-1}=1$ for all $m=1,\dots, M$, $\psi(\overline{k})=1$.
 Also one has
 $$
 \beta_N=
 \frac{a_1\ldots a_{N-1}(k_N+1)-\sum_{j=1}^{M-1}\{a_1\ldots a_{N-2j}(k_{N-2j+1}+1)-a_1\ldots a_{N-2j-1}(k_{N-2j}+1)\}
 -(k_1+1)}{a_1\ldots a_N-1}\,.
 $$
 Each summand in the curly brackets is non-negative and $k_1+1\ge 1$. Therefore $\beta_N\le 1$.
 It is equal to $1$ if and only if $k_{2m}+1=a_{2m}$ $k_{2m-1}=0$ for all $m=1,\dots, M$.
 In this case $\beta_{2m}=1$, $\beta_{2m-1}=0$ for all $m=1,\dots, M$, $\psi(\overline{k})=1$.
 
 Thus there are two possibilities: either $\psi(\overline{k})\ne 1$, $\psi(\overline{k})$
 fixes no coordinates and ${\rm age}(\psi(\overline{k}))=\sum\limits_{m=1}^{N}\beta_m=\ell(\overline{k})$
 or $\psi(\overline{k})=1$ (this happens for two basic elements), $\psi(\overline{k})$
 fixes all the coordinates and $\ell(\overline{k})=\sum\limits_{m=1}^{N}\beta_m=N/2$.
 This completes the proof for the loop case.
 \end{proof}


Let $G$ be a subgroup of $G_f$ and $\widetilde{G}$ its dual. 
The form ${\overline{x}}^{\overline{k}}dx_1\wedge\dots\wedge dx_n$ is $G$-invariant if and only if
$\psi(\overline{k})\in \widetilde{G}$. Therefore the map $\psi$ restricts to $\Omega_f^G\longrightarrow \C[\widetilde{G}]$
where $\Omega_f^G$ denotes the $G$-invariant subspace of $\Omega_f$. 

For $I \subset \{1, \ldots, N\}$, let $G_I$ be the set of $g \in G$ which fix all the coordinates from $I$ and do not fix any of the other coordinates and $\widetilde{G}_{(I)}$ the set of elements of $\widetilde{G}$ which are images under the map $\psi$ of all the monomials of the form $\bigwedge_{i \in I} x_i^{k_i} dx_i$. 

For $g\in G$ and $\widetilde{g}\in\widetilde{G}$, set $m_{g,\widetilde{g}}:=2^r$ where $r$ is the number of indices $i$ such that $f_i$ is a loop with an even number of variables and both corresponding components $g_i$ and $\widetilde{g}_i$ are equal to 1. 
Define 
\[
\widehat{m}_{g,\widetilde{g}}:=
\begin{cases}
0,\ \text{if }\widetilde{g} \not\in \widetilde{G}_{(I_g)},\\
m_{g,\widetilde{g}},\ \text{otherwise},
\end{cases}
\]
where $I_g$ is the set of indices of coordinates fixed by $g$. 

\begin{proposition} \label{prop:formula}
One has 
\begin{equation}
E(f,G)(t, \bar{t}) = \sum_{(g, \widetilde{g}) \in G \times \widetilde{G}}   \widehat{m}_{g,\widetilde{g}} (-1)^{n_g}(t \bar{t})^{{\rm age}(g)-\frac{n-n_g}{2}} \left( \frac{\bar{t}}{t} \right)^{{\rm age}(\widetilde{g})-\frac{n-n_{\widetilde{g}}}{2}}.
\end{equation}
\end{proposition}
\begin{proof}
By Remark~\ref{rem:Poincare} and Proposition~\ref{Omega to G},
we obtain from Proposition~\ref{formula} the following equation
\begin{equation}
E(f,G)(t, \bar{t}) = \sum_{I \subset I_0} \left( \left( \sum_{g \in G_I} (-1)^{|I|}(t \bar{t})^{{\rm age}(g)-\frac{n-|I|}{2}} \right) \left( \sum_{\widetilde{g} \in \widetilde{G}_{(I)}} m_{g,\widetilde{g}} \left( \frac{\bar{t}}{t} \right) ^{{\rm age}(\widetilde{g})-\frac{n-n_{\widetilde{g}}}{2}} \right) \right).
\end{equation}
\end{proof}

Now the proof of Theorem~\ref{main} is straight forward.
Indeed, one has 
\[
E(\widetilde{f},\widetilde{G})(t, \bar{t}) = \sum_{(\widetilde{g},g) \in \widetilde{G} \times G}   \widehat{m}_{\widetilde{g},g} (-1)^{n_{\widetilde{g}}} (t \bar{t})^{{\rm age}(\widetilde{g})-\frac{n-n_{\widetilde{g}}}{2}}  \left( \frac{\bar{t}}{t} \right)^{{\rm age}(g)-\frac{n-n_g}{2}}  
\]
where the coefficient $\widehat{m}_{\widetilde{g},g}$ is defined in the same way as above with the roles of $G$ and $\widetilde{G}$ interchanged. One can see that $\widehat{m}_{\widetilde{g},g}=\widehat{m}_{g,\widetilde{g}}$ and, for those pairs $(g,\widetilde{g})$ with $\widehat{m}_{\widetilde{g},g}\neq 0$, $(-1)^{n_g}=(-1)^{n-n_{\widetilde{g}}}$. Therefore
\begin{eqnarray*}
E(\widetilde{f},\widetilde{G})(t^{-1},\bar{t}) & = & (-1)^n \sum_{(\widetilde{g},g) \in \widetilde{G} \times G}   \widehat{m}_{\widetilde{g},g} (-1)^{n_g}  \left( \frac{\bar{t}}{t} \right)^{{\rm age}(\widetilde{g})-\frac{n-n_{\widetilde{g}}}{2}} (t \bar{t})^{{\rm age}(g)-\frac{n-n_g}{2}}  \\
& = &  (-1)^n E(f,G)(t, \bar{t})\, .
\end{eqnarray*}
This finishes the proof of Theorem~\ref{main}.


\end{document}